\documentclass[11pt,twoside]{preprint}
\usepackage{amsmath,amsthm,amssymb,enumerate}
\usepackage{hyperref}
\usepackage{mathrsfs}
\usepackage{breakurl}
\usepackage{tikz}

\usepackage{xcolor}
\usepackage{mhequ}
\usepackage{comment}
\usepackage{times}
\usepackage{microtype}
\usepackage[a4paper,innermargin=1in,outermargin=1.8in,bottom=1.5in,marginparwidth=1.5in,marginparsep=3mm]{geometry}

\colorlet{darkblue}{blue!90!black}
\colorlet{darkred}{red!90!black}

\usetikzlibrary{snakes}
\usetikzlibrary{decorations}
\usetikzlibrary{positioning}
\usetikzlibrary{shapes}

\colorlet{symbols}{black!50}
\definecolor{connection}{rgb}{0.7,0.1,0.1}

\tikzset{
root/.style={circle,fill=black!50,inner sep=0pt, minimum size=3mm},
        dot/.style={circle,fill=black,inner sep=0pt, minimum size=1.5mm},
        dotred/.style={circle,fill=black!50,inner sep=0pt, minimum size=2mm},
        var/.style={circle,fill=black!10,draw=black,inner sep=0pt, minimum size=3mm},
        kernel/.style={semithick,shorten >=2pt,shorten <=2pt},
        kernels/.style={snake=zigzag,shorten >=2pt,shorten <=2pt,segment amplitude=1pt,segment length=4pt,line before snake=2pt,line after snake=5pt,},
        rho/.style={densely dashed,semithick,shorten >=2pt,shorten <=2pt},
           testfcn/.style={dotted,semithick,shorten >=2pt,shorten <=2pt},
        renorm/.style={shape=circle,fill=white,inner sep=1pt},
        labl/.style={shape=rectangle,fill=white,inner sep=1pt},
        xic/.style={very thin,circle,fill=symbols,draw=black,inner sep=0pt,minimum size=1.2mm},
        xi/.style={very thin,circle,fill=blue!10,draw=black,inner sep=0pt,minimum size=1.2mm},
        xix/.style={crosscircle,fill=blue!10,draw=black,inner sep=0pt,minimum size=1.2mm},
	xib/.style={very thin,circle,fill=blue!10,draw=black,inner sep=0pt,minimum size=1.6mm},
	xie/.style={very thin,circle,fill=green!50!black,draw=black,inner sep=0pt,minimum size=1.6mm},
	xid/.style={very thin,circle,fill=symbols,draw=black,inner sep=0pt,minimum size=1.6mm},
	xibx/.style={crosscircle,fill=blue!10,draw=black,inner sep=0pt,minimum size=1.6mm},
	kernels2/.style={very thick,draw=connection,segment length=12pt},
	not/.style={thin,circle,fill=symbols,draw=connection,fill=connection,inner sep=0pt,minimum size=0.5mm},
	>=stealth,
        }

\makeatletter
\def\DeclareSymbol#1#2#3{\expandafter\gdef\csname MH@symb@#1\endcsname{\tikz[baseline=#2,scale=0.15,draw=symbols,line join=round]{#3}}\expandafter\gdef\csname MH@symb@#1s\endcsname{\scalebox{0.7}{\tikz[baseline=#2,scale=0.15,draw=symbols,line join=round]{#3}}}}
\def\<#1>{\csname MH@symb@#1\endcsname}
\makeatother

\DeclareSymbol{Xi22}{0.5}{\draw (0,0) node[xi] {} -- (-1,1) node[xi] {} -- (0,2) node[xi] {};}

\DeclareSymbol{tree}{0.5}{\draw (1,0) node[xi]{} -- (0,1);
\draw[kernels2] (-1,2) node[xi] {} -- (0,1) node[not] {} -- (1,2) node[xi] {};}

\DeclareSymbol{ladder}{0.5}{\draw[kernels2] (2,1) node[xi]{} -- (1,0) node[not] {} -- (0,1);
\draw[kernels2] (-1,2) node[xi] {} -- (0,1) node[not] {} -- (1,2) node[xi] {};}

\DeclareSymbol{Xi2}{-2}{\draw (0,-0.25) node[xi] {} -- (-1,1) node[xi] {};}
\DeclareSymbol{Xi3}{0}{\draw (0,0) node[xi] {} -- (-1,1) node[xi] {} -- (0,2) node[xi] {};}
\DeclareSymbol{Xi4}{2}{\draw (0,0) node[not] {} -- (-1,1) node[xi] {} -- (0,2) node[xi] {} -- (-1,3) node[xi] {};}
\DeclareSymbol{Xi2X}{-2}{\draw (0,-0.25) node[xi] {} -- (-1,1) node[xix] {};}
\DeclareSymbol{XXi2}{-2}{\draw (0,-0.25) node[xix] {} -- (-1,1) node[xi] {};}

\DeclareSymbol{XiIXi2}{-1}{\draw (-1,1) node[xi] {} -- (0,0) node[xi] {} -- (1,1) node[xi] {};}
\DeclareSymbol{I'Xi}{0}{\draw[kernels2] (0,0) node[not] {} -- (0,1.5) node[xi] {};}
\DeclareSymbol{IXi}{0}{\draw (0,0) node[not] {} -- (0,1.5) node[xi] {};}
\DeclareSymbol{IXi^2}{-1}{\draw[kernels2] (-1,1) node[xi] {} -- (0,0) node[not] {} -- (1,1) node[xi] {};}
\DeclareSymbol{IXi21}{-1}{\draw (0,0) -- (1,1) node[xi] {}; \draw[kernels2] (-1,1) node[xi] {} -- (0,0) node[not] {} -- (0,1.4) node[xi] {};}

\DeclareSymbol{XiX}{-2.8}{\node[xibx] {};}
\DeclareSymbol{tauX}{-2.8}{ \draw[kernels2] (0,0) node[xibx] {};}
\DeclareSymbol{Xi}{-2.8}{\node[xib] {};}
\DeclareSymbol{XiY}{-2.8}{\node[xie] {};}
\DeclareSymbol{XiZ}{-2.8}{\node[xid] {};}

\DeclareSymbol{IXi^22}{0.5}{\draw (-1,1) -- (0,2) node[xi] {};
\draw[kernels2] (-1,1) node[xi] {} -- (0,0) node[not] {} -- (1,1) node[xi] {};}
\DeclareSymbol{IXi^22b}{0.5}{\draw (1,1) -- (0,2) node[xi] {};
\draw[kernels2] (-1,1) node[xi] {} -- (0,0) node[not] {} -- (1,1) node[xi] {};}
\DeclareSymbol{IIXi^2}{-4}{\draw (0,-1.5) node[not] {} -- (0,0);
\draw[kernels2] (-1,1) node[xi] {} -- (0,0) node[not] {} -- (1,1) node[xi] {};}
\DeclareSymbol{I'IXi^2}{-4}{\draw[kernels2] (0,-1.5) node[not] {} -- (0,0);
\draw[kernels2] (-1,1) node[xi] {} -- (0,0) node[not] {} -- (1,1) node[xi] {};}

\makeatletter
\newenvironment{claim}[1][MM]
               {\list{$\bullet$}{%
  \setbox\@tempboxa\hbox{#1}\@tempdima\wd\@tempboxa%
  \setlength{\labelwidth}{\@tempdima}
  \advance\@tempdima by 1em%
  \setlength{\leftmargin}{\@tempdima}
  \setlength{\parsep}{1mm}\setlength{\itemindent}{0mm}%
  \setlength{\labelsep}{2mm}\setlength{\itemsep}{0mm}%
  \setlength{\topsep}{1mm}%
}}{\endlist}
\makeatother

\renewcommand{\theequation}{\thesection.\arabic{equation}}

\newtheorem{theorem}{Theorem}[section]
\newtheorem{lemma}[theorem]{Lemma}

\newtheorem{corollary}[theorem]{Corollary}

	\theoremstyle{definition}
\newtheorem{assumption}[theorem]{Assumption}
\newtheorem{definition}[theorem]{Definition}

\theoremstyle{remark}
\newtheorem{remark}[theorem]{Remark}

\def\scal#1{\langle #1 \rangle}
\def\db#1{[\![ #1 ]\!]}

\newcommand{\vn}[1]{{\vert\kern-0.3ex\vert\kern-0.3ex\vert #1 
    \vert\kern-0.3ex\vert\kern-0.3ex\vert}}

\newcommand{\bn}[1]{{[\kern-0.5ex] #1 
    [\kern-0.5ex]}}

\allowdisplaybreaks

\newcommand\N{\mathbb{N}}

\newcommand\bR{\mathbf{R}}

\newcommand\bone{\mathbf{1}}

\newcommand\cR{\mathcal{R}}

\newcommand\cK{\mathcal{K}}
\newcommand\cD{\mathcal{D}}
\newcommand\CD{\mathcal{D}}
\newcommand\cB{\mathcal{B}}
\newcommand\CB{\mathcal{B}}
\newcommand\cC{\mathcal{C}}
\newcommand\CC{{\mathcal{C}}}

\newcommand\cI{\mathcal{I}}
\newcommand\cJ{\mathcal{J}}

\newcommand\cS{\mathcal{S}}
\newcommand\cG{\mathcal{G}}

\newcommand\cW{\mathcal{W}}
\newcommand\cE{\mathcal{E}}

\newcommand\cP{\mathcal{P}}
\newcommand\cF{\mathcal{F}}
\newcommand\cA{\mathcal{A}}

\newcommand\scD{\mathscr{D}}
\newcommand\scB{\mathscr{B}}
\newcommand\scT{\mathscr{T}}
\newcommand\TT{\mathscr{T}}

\newcommand\scC{\mathscr{C}}

\newcommand\frL{\mathfrak{L}}

\newcommand\frs{\mathfrak{s}}

\newcommand\frI{\mathfrak{I}}

\def\eps{\varepsilon}

\def\Parts{\cP}

\def\C{\mathfrak{K}}

\newcommand{\R}{\mathbf{R}}
\newcommand{\T}{\mathbf{T}}
\newcommand{\E}{\mathbf{E}}

\def\d{\partial}
\def\id{\mathrm{id}}
\def\PPi{\boldsymbol{\Pi}}

\begin{document}

\title{A solution theory for quasilinear singular SPDEs}
\author{M\'at\'e Gerencs\'er${}^1$ and Martin Hairer${}^2$}
\institute{IST Austria, \email{mate.gerencser@ist.ac.at} \and 
Imperial College London, \email{m.hairer@imperial.ac.uk}}
\date{\today}

\maketitle

\begin{abstract}
We give a construction allowing to construct local renormalised solutions to general quasilinear
stochastic PDEs within the theory of regularity structures, thus greatly 
generalising the recent results of \cite{BDH,FGub,OWeb}. 
Loosely speaking, our construction covers quasilinear variants of all classes of equations
for which the general construction of \cite{H0,BHZ,CH} applies,
including in particular one-dimensional systems with KPZ-type nonlinearities driven by
space-time white noise. 
In a less singular and more specific case, we furthermore show that the
counterterms introduced by the renormalisation procedure are given by local functionals of the solution.
The main feature of our construction is that it allows to exploit a
number of existing results developed for the semilinear case, so that the number of additional
arguments it requires is relatively small.  
\end{abstract}

\tableofcontents

\section{Introduction}
Amidst the recent heightened interest in singular stochastic partial differential equations (SPDEs),
three different methods \cite{BDH,FGub,OWeb} have been developed
to extend the theory to quasilinear equations.
The first two of these worked with paracontrolled calculus, while \cite{OWeb} introduced a new variation of previous techniques
to treat singular SPDEs which is closer to the theories of rough paths and regularity structures,
but flexible enough to cover quasilinear variants.
For a comparison between them in terms of scope we refer the reader to the introduction of \cite{FGub},
but it should be noted that in a sense all of them deal with the `first interesting' case,
when the noise is just barely too rough for the product $a(u)\Delta u$ to make sense.
In particular, quasilinear variants of the KPZ equation, or for example the parabolic Anderson model in a generalised form in 3 dimensions, are all outside of the scope of these works.
One exception is the forthcoming work \cite{HendrikNew} which extends \cite{OWeb} to the next regime of regularity, which includes noises slightly better than
 space-time white noise in $1$ dimension (similar to the setting of our example \eqref{eq:example} below).

In the present article, we tackle this problem within the framework of regularity structures.
The generality in which we succeed in building local solution theories is, in some sense, 
optimal: loosely speaking, we show that if an equation can be solved with 
regularity structures and its solution has positive regularity,
then its quasilinear variants can also be solved (locally).
We deal with both the analytic and the probabilistic side of the theory in the sense that we show
that the general machinery developed in \cite{CH} can be exploited in order to produce random
models that do precisely fit our needs. Another major advantage of our approach is that its formulation
is such that it
allows to leverage many existing results from the semilinear situation without requiring us to reinvent
the wheel. This is why, despite its much greater generality, this article is significantly shorter than the
works mentioned above. 

The only disadvantage of our approach, compared to \cite{BDH,FGub,OWeb}, is that it is not obvious
at all \textit{a priori} why the counterterms generated by the renormalisation procedure should
be local in the solution. The reason for this is that our method relies on the introduction of
additional ``non-physical'' components to our equation, which are given by some non-local non-linear functionals
of the solution, and we cannot rule out in general that the counterterms depend on these non-physical terms.
We do however address the question of the precise form of the counterterms generated by the renormalisation
procedure in a relatively simple case where we verify that, provided that the renormalisation
constants are chosen in a specific way (which happens to be a choice that does still allow to show convergence
of the underlying renormalised model, although it differs in general from the BPHZ renormalisation introduced
in \cite{BHZ}), 
all non-local contributions
cancel out exactly. The reason for a lack of general statement is that the 
algebraic machinery developed in \cite{CH,BCCH}, which allows to show that counterterms are always local
in the semilinear case,
does not appear to be applicable in a simple way. 
However, we do expect that this is something
that could be addressed in future work.

The concrete example we consider is a slightly regularised version of the quasilinear variant
of the generalised KPZ equation, which formally reads as
\begin{equ}\label{eq:example}
(\partial_t-a(u)\partial_x^2)u=F_0(u)(\partial_x u)^2+F_1(u)\xi,\quad\text{on }(0,1]\times\T\;,\qquad u(0,\cdot) = u_0
\end{equ}
where $\xi$ is a translation invariant Gaussian noise on $\R\times\T$ with covariance function $\mathscr{C}$ satisfying $|\scC(t,x)|\leq(|t|^{1/2}+|x|)^{-3+\nu}$
for some $\nu>0$,
$u_0\in\cC^{\bar\nu}$ for some $\bar\nu>0$,
$a$ is a smooth function taking values in $\C$ for some compact $\C\subset (0,\infty)$,
and $F_0$ and $F_1$ are smooth functions.
The quasilinear equations considered in previous works \cite{BDH,FGub,OWeb} correspond to situations where $\nu>1/3$, $F_0=0$.
Let us take a compactly supported nonnegative symmetric (under the involution $x \mapsto -x$) 
smooth function $\rho$ integrating to 1,
set $\rho^\eps(t,x)=\eps^{-3}\rho(\eps^{-2} t,\eps^{-1}x)$,
and define $u^\eps$ as the classical solution of
\begin{equs}\label{eq:example reno}
(\partial_t-a(u^\eps)\partial_x^2)u^\eps
&=F_0(u^\eps)(\partial_x u^\eps)^2+F_1(u^\eps)(\rho^\eps\ast\xi)
\\
&\quad
-C^\eps_{a(u^\eps)}(aF_1'F_1-a'F_1^2+F_1^2F_0)(u^\eps),\qquad u(0,\cdot) = u_0
\end{equs}
where $C^\eps_c$ is some smooth function of $c\in\C$.
We then have the following (renormalised) well-posedness result for equation \eqref{eq:example},
which will be proved in Section \ref{sec:example}.
\begin{theorem}\label{thm:example}
There exist deterministic smooth functions $C^\eps_\cdot$ such that
for all $u_0\in\cC^{\bar\nu}$
there exists a random time $\tau>0$ such that 
$u^\eps$ converges in probability in $\cC([0,\tau]\times\T)\cap\cC^{1/2}_{loc}((0,\tau]\times\T)$ to a limit $u$.
Furthermore, with a suitable choice of $C^\eps_\cdot$, one can ensure that the limit $u$ is independent of $\rho$.
\end{theorem}

\begin{remark}
We would like to stress again that we only need the condition $\nu > 0$ in order to guarantee that
the counterterms created by the renormalisation of the underlying model are local functions of $u$.
The rest of the argument works down to $\nu > -{1\over 2}$ (including in particular the case of space-time 
white noise), at which point the conditions of \cite[Thm~2.14]{CH} are
violated and one expects a qualitative change of the scaling behaviour of the solution. 
Similarly, we consider a scalar equation driven by a single noise purely for the sake of notational 
convenience. The exact same proof also applies for example to systems of the type
\begin{equ}
\partial_t u_i= a_{ij}(u)\partial_x^2u_j + F_{ijk}^{(2)}(u)(\partial_x u_j)(\partial_x u_k)+
F_{ij}^{(1)}(u)(\d_x u_j) + F_{ij}^{(0)}(u)\xi_j\;,
\end{equ}
with $a$ taking values in some compact set of strictly positive definite symmetric matrices and 
implicit summation over $j,k$.
\end{remark}

The structure of the remainder of this article goes as follows. In Section~\ref{sec:equivalent}, we 
first give an equivalent formulation of a general quasilinear SPDE which is the main remark this
article is based on. The main purpose of this reformulation is to write the equation in integral
form in a way that resembles the mild formulation for semilinear problems. In particular, this
can be done in such a way that the product $a(u)\cdot \d_x^2 u$ never appears and is replaced instead
by seemingly more complicated terms that however exhibit better scaling / regularity properties.
In Section~\ref{sec:structure}, we then show how to build a suitable regularity structure allowing to 
formulate the fixed point problem derived in Section~\ref{sec:equivalent}. This is very similar to what
is done in \cite{H0,BHZ} with the unusual twist that each symbol represents an \textit{infinite-dimensional}
subspace of the resulting regularity structure, rather than a one-dimensional one.
The formulation of the fixed point problem is then done in Section~\ref{sec:4}. Finally, 
we treat a concrete example in Section~\ref{sec:example}, where we also verify ``by hand'' that in this case
the renormalisation procedure does indeed only produce local counterterms.

\subsection*{Acknowledgements}

{\small
MH gratefully acknowledges support by the Leverhulme Trust and by an ERC consolidator grant, project 615897.
}

\section{An equivalent formulation}
\label{sec:equivalent}

The main observation on which this article builds is that, at least for smooth drivers / solutions, 
a quasilinear equation is 
equivalent to another equation whose principal (smoothing) part does make sense even in the limit
when the driving noise is taken to be rough.
The `right-hand side' of this new equation may however exhibit ill-defined products
(sometimes even if the original right-hand side did not), but that situation is already closer to the 
ones that the theory of \cite{H0} was developed for.

To describe this alternative formulation, we restrict our attention to the case of perturbations of the 
heat equation on the one-dimensional torus $\T$, but it is straightforward to generalise this to
other situations.
In this case, one wants to solve initial value problems of the type
\begin{equ}\label{eq: 00}
(\partial_t-a(u)\partial_x^2)u=F(u,\xi)\quad\text{on }(0,1]\times\T\;,\qquad u(0,\cdot) = u_0\;,
\end{equ}
where $a$ is a smooth function taking values in $\C$ for some compact $\C\subset (0,\infty)$, $F$ is a subcritical (in the sense of \cite{H0,BHZ}) local nonlinearity, and $\xi$ is a noise term, which for \eqref{eq: 00} to make sense, is assumed to be smooth for the moment.

\begin{remark}\label{rem:powers}
Assuming that we are interested in noises $\xi \in \CC^{\alpha-2}$ for 
$\alpha \in (0,1)$, so that potential solutions are expected to be of class $\CC^\alpha$, 
$F$ being subcritical is equivalent to assuming that it is of the form
\begin{equ}
F(u,\xi) = F_0(u,\d_x u) + F_1(u)\xi\;,
\end{equ}
where $F_0\colon \R^2 \to \R$ and $F_1\colon \R \to \R$ are smooth functions
and the dependence of $F_0$ in its second argument is polynomial of degree strictly
less than $(2-\alpha)/(1-\alpha)$.
\end{remark}

It will be convenient to write the equation in a more `global' way: setting
$f = \bone_{t>0}F(u,\xi)+\delta \otimes u_0$, where $\delta$ is the Dirac mass at time 0 and both $F$ and $u_0$
are extended periodically to all of $\R$, \eqref{eq: 00} is equivalent to
\begin{equ}\label{eq: 01}
(\partial_t-a(u)\partial_x^2)u=f \quad\text{on } (-\infty,1] \times\R.
\end{equ}
For $c>0$, denote by $P(c,\cdot)$ the Green's function 
of $\partial_t-c\partial_x^2$ on $\T$. Note that $P$ is smooth as a function of $c$ away from
the origin and one has the identity
\begin{equ}\label{eq:P c derivative}
\tfrac{\partial^\ell}{\partial c^\ell}P(c,\cdot)
= \partial_x^{2\ell} \underbrace{P(c,\cdot)\ast\cdots\ast P(c,\cdot)}_{\ell+1\text{ times}}\;,
\end{equ}
where the convolutions are in space-time.
Introduce operators $I_\ell^{(k)}$ acting on smooth functions $b$ and $f$ by 
\begin{equ}\label{eq: I}
I_\ell^{(k)}(b,f)(z) = \int (\d_x^k \d_c^\ell P)(b(z),z-z')f(z')\,dz'\;.
\end{equ}
We will also use the shorthands $I = I^{(0)}_0$, $I' = I^{(1)}_0$, $I_1=I_1^{0}$, etc.
Note that these operators are linear in their second argument, but not in the first.
Note also that, by a simple integration by parts, one has the identities
\begin{equ}
I_\ell^{(k+m)}(b,f)(z) =
I_\ell^{(k)}(b,\d_x^m f)(z)\;.
\end{equ}
Even though $I(b,f)$ is of course not the same as the solution map to $(\partial_t-b\partial_x^2)u=f$
if $b$ is non-constant, 
it turns out that solving \eqref{eq: 01} is equivalent to solving
an equation of the type
\minilab{e:system}
\begin{equ}\label{eq: mod}
u=I(a(u), \hat f),
\end{equ}
where $\hat f =\bone_{t>0}\hat F(u,\xi)+\delta \otimes u_0$
for some modified (non-local) nonlinearity $\hat F$. 
Since $I$ does make perfect sense for arbitrary $b\in\cC^{0+}$ and $f\in\cC^{-2+}$ (which are the expected regularities of the coefficient and the right-hand side, respectively, even in the limit), this moves all the ill-defined terms into 
the definition of $\hat F$.

Verifying the equivalence is elementary as long as all functions involved are smooth: suppose
that $u$ satisfies \eqref{eq: mod} and apply $\partial_t-a(u)\partial_x^2$ to both sides of this equation.
Denoting the expression $(\partial_t-a(u)\partial_x^2)u$ by $f$, one then has 
\begin{equs}
f&=\hat f+\big((\partial_t-a(u)\partial_x^2)a(u)\big)I_1(a(u),\hat f)
\\
&\quad-a(u)|\partial_x(a(u))|^2I_2(a(u),\hat f)
-2a(u)\partial_x(a(u))I_1(a(u),\d_x \hat f) \label{e:calculation} \\
&=\hat f+ a'(u) f I_1(a(u),\hat f) - (aa'')(u) (\d_x u)^2 I_1(a(u),\hat f)
\\
&\quad-(a (a')^2)(u)|\partial_x u|^2\, I_2(a(u),\hat f)
-2 (aa')(u)\partial_x u \,I_{1}(a(u),\hat f) \;.
\end{equs}
One can rearrange the above as a fixed point equation for $\hat f$ by writing it as
\begin{equs}\label{eq: mod nl2}
\hat f &= (1 - a'(u) I_1(a(u),\hat f)) f + (aa'')(u)|\partial_x u|^2I_1(a(u),\hat f) \\
&\qquad+ (a (a')^2)(u)|\partial_x u|^2\, I_2(a(u),\hat f) + 2 (aa')(u)\partial_x u \,I_1'(a(u), \hat f)\;.
\end{equs}
Now we note that since $f$ is of the form 
$f(t,x) = \bone_{t>0}(F(u,\xi))(t,x) + \delta(t) u_0(x)$, where
$F=F(u,\xi)$ is a $\cC^1$ function on $[0,1]\times\T$, 
we can look for solutions to this fixed point problem that are also of
the form $\hat f(t,x) =\bone_{t>0} \hat F(t,x) + \delta(t) u_0(x)$.
To see this, define the operator
\begin{equ}\label{eq: Ihat}
\hat I_\ell^{(k)}(b,g)(z) = I_\ell^{(k)}(b,\delta \otimes g)(z) = \int (\d_c^\ell P_t)(b(z),x-x')\d_x^kg(x')\,dx'\;,
\end{equ}
where we use the convention $z=(t,x)$. (This is really how all the terms involving $\delta$ should be
interpreted in \eqref{eq: mod}--\eqref{eq: mod nl2}.)
The function $I_1(a(u),\hat F)$ is continuous and vanishes at time $0$, as does $\hat I_1(a(u),u_0)$ 
for any $u_0$ of strictly positive regularity, 
as one can see from \eqref{eq:P c derivative} for example. 
Thus \eqref{eq: mod nl2} can be written as a fixed point problem for $\hat F$:
\minilab{e:system}
\begin{equs}\label{eq: mod nl3}
\hat F &= \big(1 - a'(u) \,( I_1(a(u),\hat F) + \hat I_1(a(u),u_0)\big) \bone_{t>0}F(u,\xi) \\
&\qquad+ (aa'')(u)|\partial_x u|^2\, ( I_1(a(u),\hat F) +\hat I_1(a(u),u_0) )\\
&\qquad+ (a (a')^2)(u)|\partial_x u|^2\, (I_2(a(u),\hat F)+\hat I_2(a(u),u_0)) \\
&\qquad+ 2 (aa')(u)\partial_x u \,( I_1'(a(u), \hat F)+\hat I_1'(a(u),u_0) )\;.
\end{equs}
If $u_0$ is sufficiently smooth, say $\cC^3$,
one can write the system \eqref{e:system} as a fixed point problem
\begin{equ}\label{eq: mod nl4}
(u,\hat F)=\cA_{u_0,\xi}(u,\hat F),
\end{equ}
where $\cA_{u_0,\xi}$ is a contraction on a ball of $(\cC_0^{4/3}\times\cC_0^{-1/3})(\T_t)$ for small times $t$, where $\T_t=(-\infty,t]\times\T$ and $\cC_0^\alpha$ is consists of the space-time $\alpha$-H\"older regular distributions that vanish for negative times.
Indeed, for the first coordinate of $\cA_{u_0,\xi}$ this is immediate from classical Schauder estimates.
For the second coordinate it suffices to notice that thanks to Remark \ref{rem:powers}
the right-hand side of \eqref{eq: mod nl3} is locally Lipschitz continuous from $(\cC_0^{4/3}\times\cC_0^{-1/3})(\T_t)$ to $\cC_0^0(\T_t)$,
which is in turn embedded into $\cC_0^{-1/3}(\T_t)$, and the norm of this injection is proportional to a positive power of $t$.

Using a version of this argument with temporal weights, it is straightforward to 
show that \eqref{eq: mod nl4} admits a unique local solution $(u,\hat F)$.
Furthermore, the preceding calculations show that,
as long as the function $g$ given by 
\begin{equ}
g =a'(u) \,( I_1(a(u),\hat F) + \hat I_1(a(u),u_0)\;,
\end{equ}
is strictly smaller than $1$, one does indeed have
\begin{equ}\label{eq:10}
\big((\partial_t-a(u)\partial_x^2)u\big)(t,x)=(F(u,\xi))(t,x)+\delta(t)u_0(x)\;.
\end{equ}
Since, by the same reasoning as above, $g$ is continuous and $g\rightarrow0$ as $t\rightarrow0$, 
the claim follows.
Moreover, if $|u_0|_{\cC^{0+}}\leq C$ for some constant $C$, then for any fixed $t_1>0$, the time $t_0:=\sup\{t\in[0,t_1]:\,|g(t,x)|<1\,\forall x\in\T\}>0$ can be bounded
from below in terms of the $\cC^{-2+}([-1,t_1]\times \T)$ norm of $\hat F$.
A reasonable solution theory of \eqref{eq: mod nl4} -- which of course will require a renormalisation of the right-hand side of \eqref{eq: mod nl4} -- is expected to imply that for some $t_1>0$, this norm is uniformly bounded
over a given family of smooth approximations of the `true' noise $\xi$, and hence $t_0$ is uniformly 
bounded away from 0.

It is not difficult to convince oneself that this argument is quite robust.
For example, if in \eqref{eq: 00} the operator $\partial_x^2$ is replaced by $\partial_x^{2k}$ for some $k\in\mathbb{N}$, or if in higher dimensions, $a(u)$ is matrix-valued and acts on the Hessian of $u$ in a non-degenerate way, similar arguments show the analogous equivalence, of course with $I$ built from a suitably modified family of parametrised kernels.
It therefore suffices to solve equations of the type \eqref{e:system}, which one can do using the framework of regularity structures as we will demonstrate in the remainder of this article.

\section{Regularity structures with continuous parameter-dependence}
\label{sec:structure}

It should be clear at this point that we would like to encode in our regularity structure the integration against all kernels 
$P(c,\cdot)$, as well as some of their derivatives with respect to $x$ and $c$.
Since there is a continuum of them and since one wants to have some control over the dependence on $c$, this 
requires a  modification of the construction in \cite{H0}.

The starting point of our construction however is very similar to that given in \cite{H0,BHZ} and we quickly recall 
it here, mainly to fix notations.
We fix a dimension $d\geq1$ and on it, a scaling $\frs\in\N^d$.
We assume that we are given a finite index set $\frL = \frL_+ \sqcup \frL_-$
as well as a map $\alpha \colon \frL \to \R\setminus\{0\}$ which is positive on $\frL_+$ and negative on $\frL_-$.
We build from this a set of symbols $\cF$ by decreeing it to be the minimal set satisfying the following properties. 
\begin{claim}
\item There are symbols $\Xi_i$ and $X^k$
belonging to $\cF$ for all $i\in \frL_-$ and any $d$-dimensional multiindex $k$. We also write $\bone=X^0$
and $X_i = X^{e_i}$ with $e_i$ the $i$th canonical basis vector.
\item For any $\tau,\tau',\tau''\in\cF$, one also has $\tau\tau'\in\cF$, and $\tau(\tau'\tau'')$ and $(\tau\tau')\tau''$ are identified. We also identify $X^k X^\ell$ with $X^{k+\ell}$, $X^k \tau$ with $\tau X^k$, and $\bone \tau$ with $\tau$. 
\item For any $j\in \frL_+$, any $d$-dimensional multiindex $k$ and any $\tau\in\cF$, one also has 
a symbol $\cI_j^{(k)}\tau\in\cF$.
\end{claim}

\begin{remark}
It is important here that unlike \cite{H0,BHZ} we do \textit{not} identify $\tau \bar \tau$ with $\bar \tau \tau$! 
The freedom to leave these as separate symbols will be very convenient later on. 
\end{remark}

We naturally associate degrees $|\cdot|$ to these symbols by postulating that
\begin{equ}[e:degree]
|\Xi_i|=\alpha_i,\quad |X^k|=|k|_\frs,\quad|\tau\bar\tau|=|\tau|+|\bar\tau|,
\quad|\cI_j^{(k)}\tau|=|\tau|+\alpha_j - |k|_\frs\;.
\end{equ}
We then consider the map $\cG \colon \cF \to \Parts(\cF)$ (the set of all subsets of $\cF$) defined as the minimal
map ($\Parts(\cF)$ being ordered by inclusion) satisfying the following properties.
\begin{claim}
\item One has $\tau \in \cG(\tau)$ for every $\tau \in \cF$, and one has 
$\cG(\bone) = \{\bone\}$, $\cG(X_i) = \{\bone, X_i\}$, $\cG(\Xi_i) = \{\Xi_i\}$.
\item One has $\cG(\tau \bar \tau) = \{\sigma \bar \sigma\,:\, \sigma \in \cG(\tau)\;\&\; \bar \sigma \in \cG(\bar \tau)\}$.
\item One has $\cG(\cI^{(j)}_\ell \tau) = \{\cI^{(j)}_\ell \sigma\,:\, \sigma \in \cG(\tau)\}\cup \{X^k\,:\, |k| < |\cI_\ell^{(j)}\tau|\}$.
\end{claim}
The motivation for this definition is that these properties of the set $\cG(\tau)$ guarantee
that every element of the structure group associated to our regularity structure as in \cite{BHZ} 
maps any given symbol $\tau$ into the linear span of $\cG(\tau)$.
This allows us to give the following definition of a subcritical set $\cW$, which one should think of
any subset of $\cF$ that generates an actual regularity structure (one in which the set of possible
degrees is locally finite and bounded from below).

\begin{definition}
A subset $\cW \subset \cF$ is said to be \textit{subcritical} if it satisfies the following properties.
\begin{claim}
\item If $\tau \in \cW$, then $\cG(\tau) \subset \cW$.
\item For every $\gamma \in \R$, the set $\{\tau \in \cW \,:\, |\tau| <\gamma\}$ is finite.
\end{claim}
It is said to be \textit{normal} if, whenever $\tau \bar\tau \in \cW$,
one has $\{\tau,\bar \tau\} \subset \cW$ and, whenever $\cI^{(j)}_\ell\tau \in \cW$,
one has $\tau \in \cW$.
\end{definition}

As shown in \cite{H0,BHZ}, every locally subcritical  stochastic PDE (or system thereof) naturally
determines a normal subcritical set $\cW$.
From now on, we consider $\cW$ to be fixed and we only ever consider elements $\tau \in \cW$.

\subsection{A regularity structure}

In \cite{H0,BHZ},
one then constructs a regularity structure by taking the vector space $\scal{\cW}$ generated by $\cW$ as the structure
space (graded by the notion of degree given in  \eqref{e:degree}) and endowing it with a suitable structure group.
In our situation, to encode parameter dependence, we instead assign to each element of $\cW$ a
typically infinite-dimensional subspace of the structure space.
In order to encode this,
we first define the `number of parameters' $[\tau]$ in a symbol $\tau$ recursively by setting
\begin{equ}{}
[X^k]=[\Xi_i]=0,\quad[\tau\bar\tau]=[\tau]+[\bar\tau],\quad[\cI_j^{(k)}\tau]=[\tau]+1.
\end{equ}
\begin{remark}
One could in principle encode some parametrisation of the noises as well, by setting $[\Xi_i]=1$, or we
could even allow the number of parameters to depend on the element of $\frL$ we consider.
Since this generality is not used in the sequel, we refrain from doing so here.
\end{remark}

We also assume that we
are given a real Banach space $\CB$ and we write $\CB_k$ for the $k$-fold tensor product of $\CB$
with itself, completed under the projective cross norm.
In particular, we have a canonical dense embedding of $\CB_k\otimes \CB_\ell$ into $\CB_{k+\ell}$.
We also use the convention $\CB_0=\R$.
Given a normal subcritical set of symbols $\cW$, we then construct
a regularity structure from it in such a way 
that each symbol $\tau \in\cW$ determines an infinite-dimensional subspace $T_\tau$ of the structure 
space $T$, isometric to $\CB_{[\tau]}$. To wit, we set
\begin{equ}\label{eq:T}
T=\bigoplus_{\alpha} T_\alpha\;,\quad
T_\alpha:=\bigoplus_{|\tau| = \alpha} T_\tau\;,\quad
T_\tau := \CB_{[\tau]}\otimes\scal{\tau}\;,\quad
\end{equ}
and equip the spaces $T_\alpha$ with their natural norms. Here, we wrote $\scal{\tau}$ for the one-dimensional
real vector space with basis $\tau$. (By the definition of subcriticality, there are
only finitely many symbols $\tau$ with $|\tau| = \alpha$, so that the $T_\alpha$ are naturally endowed with
a Banach space structure. This is not the case for $T$ itself though, but we view it as a topological vector 
space in the usual way.) For $\tau$ with $[\tau] = 0$,
we also identify $T_\tau$ with $\scal{\tau}$.

The space $T$ comes equipped with a number of natural operations. 
For every $i \in \{1,\ldots,d\}$, we have an abstract differentiation $D_i$
acting on $\scal{\cF}$
by setting $D_i X_j = \delta_{ij}\bone$, $D_i \bone = 0$, 
$D_i \cI_j^{(k)}(\tau) = \cI_j^{(k+e_i)}(\tau)$, and then extending it to all other symbols
by enforcing that Leibniz's rule holds. 
For any $\tau\in\cW$ such that $D_i \tau \in \scal{\cW}$ and any $b\in \CB_{[\tau]}$, we then set
\begin{equ}
\scD_i (b\otimes\tau) =b\otimes D_i\tau\;.
\end{equ}
This is indeed well-defined since $D_i \tau$ is a linear combination of elements $\sigma$
with $[\sigma]=[\tau]$.
Similarly, for the abstract product, whenever $\tau,\bar\tau,\tau\bar\tau\in\cW$, 
$b\in \CB_{[\tau]}$, $\bar b\in \CB_{[\bar\tau]}$, we set
\begin{equ}
(b\otimes\tau) (\bar b\otimes\bar\tau) = (b \otimes \bar b)\otimes \tau\bar \tau\;,
\end{equ}
with $b \otimes \bar b$ interpreted as an element of $\CB_{[\tau\bar \tau]}$.
(Here it is convenient that $\tau\bar \tau$ and $\bar \tau\tau$ aren't identified since it avoids being forced
to deal with symmetric tensor products.)
Finally, we have a large number of abstract integration operators: 
for any $j\in \frL_+$ and any $b\in \CB$, 
a map $\cI_j^{(k),b}$ is defined as the linear extension of
\begin{equ}
\cI_j^{(k),b} (\bar b\otimes\tau) := (b\otimes \bar b)\otimes  \cI_j^{(k)}\tau\;,
\end{equ}
defined on those $T_\tau$ for which $\cI_j^{(k)}\tau \in \cW$.

So far we have not addressed the structure group at all, but its inductive construction as in \cite[Thm~5.14]{H0} is virtually identical in our setting.
More precisely, as in \cite[Defs~4.6 \&\ 5.25]{H0}
the group $G$ consists of those continuous linear operators $\Gamma \colon T \to T$
satisfying the following properties.
\begin{claim}
\item For any $\alpha \in \R$, one has $(\Gamma - \id) \colon T_\alpha \to T_{<\alpha}$.
\item One has $\Gamma \bone = \bone$, $\Gamma \Xi_i = \Xi_i$ and there are constants $c_i$ such that 
$\Gamma X_i = X_i - c_i \bone$.
\item For any two symbols $\tau$, $\bar \tau$ in $\cW$ such that $\tau\bar\tau \in \cW$ and any 
$a \in T_\tau$, $\bar a \in T_{\bar \tau}$, one has $\Gamma(a  \bar a) = (\Gamma a)  (\Gamma \bar a)$.
\item For any $\tau \in \cW$ such that $D_i \tau \in \scal{\cW}$ and any $a \in T_\tau$, one has
$\Gamma \scD_i a = \scD_i \Gamma a$.
\item For any $\tau \in \cW$ such that $\cI_\ell^{(k)} \tau \in \cW$, any $a \in T_\tau$
and any $b \in \CB$, one has 
\begin{equ}
(\Gamma \cI_\ell^{(k),b} - \cI_\ell^{(k),b}\Gamma ) a \in \scal{\{X^k\,:\, k \in \N^d\}}\;.
\end{equ}
\end{claim}
As in \cite{H0}, on can show that this is indeed a group. The definitions of $G$ and $\cG$ also guarantee that,
for any $\tau \in \cW$, any $\Gamma\in G$ maps $T_\tau$ to $\bigoplus_{\sigma \in \cG(\tau)}T_\sigma$,
which is indeed a subspace of $T$ by the assumption on $\cW$.

From now on, we write $\TT$ for the regularity structure with structure space $T$ and structure group $G$
given as above with the specific choice
\begin{equ}[e:defBn]
\CB= \cC^{N}(\C)\;,
\end{equ}
for some compact parameter space $\C$, which is assumed for simplicity
to be a subset of a $\R^{d_1}$, as well as some 
sufficiently large $N  > 0$ to be determined later.

\subsection{Admissible models}

We assume henceforth that $\frL_+$ and $\frL_-$ are singletons, and therefore omit the lower indices in $\Xi$, $\cI$. 
This is purely for the sake of notational convenience, this section immediately extends to the general case. 
We also omit $k$ in $\cI^{(k)}$ and $\cI^{(k),b}$ if $k=0$
and we set $\alpha=|\Xi|$ and $\beta=|\cI\tau|-|\tau|$.

\begin{assumption}\label{asn: K}
We are given a family of kernels $(K^{(c)})_{c\in \C}$, which,
along with their derivatives with respect to $c$ up to any finite order,
are uniformly compactly supported and $\beta$-smoothing in the 
sense of \cite[Ass.~5.1]{H0}.
\end{assumption}

\begin{remark}
Think of $K^{(c)}$ as the heat kernel with parameter $c$ as in \eqref{eq:P c derivative}.
These are of course not compactly supported and do not satisfy \cite[Ass.~5.4]{H0}. However, it is always possible
to choose kernels $K^{(c)}$ satisfying these properties and such that 
their projection onto $\Lambda=[-1,\infty) \times \T^{d-1}$ (obtained by adding all integer translates)
agrees with the heat kernel,
so that if $\eta$ is a distribution supported on $\R_+ \times \T^{d-1}$,
one has $K^{(c)} \ast \eta = P(c,\cdot)\ast \eta$ on $[0,1]\times \T^{d-1}$.
One then also has
\begin{equ}\label{eq:approx green}
(\partial_t-c\Delta)K^{(c)}=\delta_0+f^{(c)},
\end{equ}
where $\delta_0$ is the Dirac mass at the space-time origin and $f^{(c)}$ are
compactly (and away from the origin) supported smooth functions, depending smoothly on $c$.
\end{remark}

Furthermore, for any distribution $\zeta$ on $\C$, any $c\in\C$, and any $\ell\in\mathbb{N}^{d_1}$,
we write
\begin{equ}
K^{\zeta}:=\zeta(K^{(\cdot)}),
\quad\quad
K^{c;\ell}:=K^{\partial^\ell\delta_c}.
\end{equ}
By the assumption on $K$, $K^\zeta$ is also $\beta$-smoothing in the sense of \cite[Ass.~5.1]{H0}
and, when considering its decomposition $K^\zeta=\sum_{n\geq0} K^\zeta_n$,
one has a bound of the type $|D^k K^\zeta_n|\lesssim 2^{n(|\frs|-\beta+|k|_\frs)}|\zeta|_{\cC^{-N_0}}$
for any fixed $N_0 > 0$.

As our notation suggests, we want the maps $\cI^{\zeta}$ to correspond to integrations against the kernels $K^\zeta$,
which is encoded in the definition of admissibility in the present setting.
\begin{definition}\label{def:admiss}
In the above setting, a model $(\Pi,\Gamma)$ is admissible for $\TT$ if, for all $\alpha\in A$, $\tau\in \cW$, such that $|\tau|=\alpha$ and $\cI\tau\in\cW$,
for all $\sigma\in T_{\tau}$, $\zeta\in\cC^{-N}$,
$x\in\R^d$, 
and $\varphi\in C_0^\infty$ the following identity holds
\begin{equ}
(\Pi_x\cI^{\zeta} \sigma)(\varphi)=
(K^{\zeta}\ast\Pi_x\sigma)(\varphi)-\int\sum_{|k|_\frs\leq \alpha+\beta}\frac{(y-x)^k}{k!}\Pi_x\sigma \big(D^kK^{\zeta}(x-\cdot)\big)\varphi(y)\,dy.
\end{equ}
\end{definition}
One can define the maps $\cJ^\zeta$ as in \cite{H0}, with $K$ therein replaced by $K^\zeta$.
With this notation the second term on the right-hand side above can be also written as $(\Pi_x\cJ^\zeta(x)\sigma)(\varphi)$.

In the following we borrow the notations $\psi_x^\lambda$, $\|\Pi-\Pi'\|_{\gamma;\, B}$, $\|\Gamma-\Gamma'\|_{\gamma;\,B}$ from \cite{H0}, and denote by $\scB$ the set $\cB_{-\lfloor\alpha\rfloor}$ of test functions
considered there. (This is in order to prevent confusion with the scale of spaces $\cB_k$.)
In fact, the lower indices in the norms of the models will usually be omitted for brevity,
since the dependence on them does not play any role in our discussion.

\subsection{Constructing models}
In principle, if one has a sufficiently robust way of building a model (or a family of models with some continuity properties)
for the (usual) regularity structure determined by $\cW$, one can also build an admissible
model for the parametrised regularity structure $\scT$.
Such a `robust way of building models' is developed in great generality in \cite{CH}.
To be self-contained regarding the assumptions required to recall some of its results,
we restrict our attention to the Gaussian case and refer the reader to \cite{CH} for more general noises that fit in the theory.

\begin{assumption}\label{asn:noise}
Suppose we are given a centred, Gaussian, translation invariant, $\cS'(\R^d)$-valued random variable $\xi$,
such that there exists a distribution $\scC$ whose singular support is contained in $\{0\}$,
which satisfies
\begin{equ}
\E\big(\xi(f)\xi(g)\big)=\scC\Big(\int f(z-\cdot)g(z)\,dz\Big)
\end{equ}
for all test functions $f,g\in\cS(\R^d)$.
Writing $z\mapsto\scC(z)$ for the smooth function that determines $\scC$ away from $0$, 
it is furthermore assumed that any test function $g$ satisfying $D^kg(0)=0$ for all multiindex $k$ with 
$|k|_\frs<-|\frs|-2\alpha$, one has
\begin{equ}
\scC(g)=\int\scC(z)g(z)\,dz.
\end{equ}
Finally, there exists a $\kappa>0$, such that for all multiindex $k$
\begin{equ}
\sup_{0<|z|_\frs\leq 1}|D^k\scC(z)|\,|z|_\frs^{|k|_\frs-2\alpha-\kappa}<\infty.
\end{equ}
\end{assumption}
The final assumption on $\cW$ is what is referred to as super-regularity in \cite{CH}, which in the present
setting reads as follows.
Define, similarly to $[\cdot]$, the `number of noises' $\db{\tau}$ in a symbol $\tau$ recursively by
\begin{equ}{}
\db{X^k}=0,\quad \db{\Xi}=1,\quad\db{\tau\bar\tau}=\db{\tau}+\db{\bar\tau},\quad\db{\cI^{(k)}\tau}=\db{\tau}.
\end{equ}
\begin{assumption}\label{asn:supreg}
All $\tau\in\cW$ with $[\![\tau]\!]\geq2$ satisfy
$|\tau|>\alpha$ and $|\tau|>-|\frs|/2$.
If $\db{\tau}\geq3$, then also $|\tau|>-(|\frs|+\alpha)$, while if $\db{\tau}=2$, then also $|\tau|>-2(|\frs|+\alpha)$ holds. 
\end{assumption}

Take, as in the introduction, a compactly supported nonnegative symmetric smooth function $\rho$ integrating to 1,
and set $\rho^\eps(t,x)=\eps^{-|\frs|}\rho(z_1\eps^{-\frs_1},\ldots,z_d\eps^{-\frs_d})$.
Under the above assumptions, we wish to construct a family of admissible models $(\hat \Pi^\eps,\hat \Gamma^\eps)_{\eps\in[0,1]}$
that is continuous in a suitable sense in the $\eps\rightarrow0$ limit, and which satisfy
$\hat \Pi_z^\eps\Xi=\rho^\eps\ast\xi$
(here and below we use the natural convention of $\rho^0\ast$ denoting the identity).

Denote $N_0=N+d_1+1$,
for a finite set $\tilde B\subset \cC^{-N_0}$ let $S_{\tilde B}$ be the set of `simple' elements of the form
$a=\big(\bigotimes_{i=1}^{[\tau]}\zeta_i\big)\otimes\tau$, with $\tau\in\cW$, $\zeta_i\in \tilde B$,
and let $S=\bigcup_{\tilde B} S_{\tilde B}$ (notice that $S\nsubseteq T$ since one has $N_0 > N$!).
Any $a\in S_{\tilde B}$ can be mapped to an element $\iota(a)$ of the structure space $T_{\tilde B}$ for
the regularity structure $\scT_{\tilde B}$ built from $\frL_+=\tilde B$, $\alpha(\frL_+)=\{\beta\}$, by setting recursively
\begin{equ}
\iota(\Xi)=\Xi,\quad\iota(X^k)=X^k,\quad\iota(a\bar a)=\iota(a)\iota(\bar a),\quad
\iota(\cI^{(\ell),\zeta}a)=\cI_\zeta^{(\ell)}\iota(a).
\end{equ}
Let $Z^\eps = (\Pi^\eps,\Gamma^\eps)$ for ${\eps\in[0,1]}$ be the family of BPHZ models for $\scT_{\tilde B}$ as constructed in \cite{BHZ,CH},
which satisfy in particular $\Pi_z^\eps\Xi=\rho^\eps\ast\xi$.
One can then define the random distributions
\begin{equ}[e:constr]
\bar\Pi^\eps_x a:=\Pi^\eps_x\iota(a),
\end{equ}
for $a\in S$. Note that formally the right-hand side also depends on $\tilde B$
(the regularity structure $\scT_{\tilde B}$ in which $\iota(a)$ takes values depends on it,
as well as the model $Z^\eps$), but our construction is such that different choices of $\tilde B$
yield the same right hand side in \eqref{e:constr}.
By \cite{CH}, the random fields $\bar\Pi_x$ satisfy the bounds
\begin{equs}[eq:model 1st bounds]
\sup_{0\neq a\in S}|a|^{-p}\E\sup_{x,\lambda,\psi}\lambda^{-p|\tau|}|(\bar\Pi^\eps_x a)(\psi_x^\lambda)|^p&\lesssim 1,
\\
\sup_{0\neq a\in S}|a|^{-p}\E\sup_{x,\lambda,\psi}\lambda^{-p|\tau|}|((\bar\Pi_x^\eps-\bar\Pi_x^0) a)(\psi_x^\lambda)|^p&\lesssim \eps^{p\theta},
\end{equs}
with some $\theta>0$, where here and below the second supremum is taken over $x$ in some
compact set, $\lambda\in(0,1]$, and $\psi\in\scB$.
The random field $\bar\Pi$ can then be turned into an admissible model for $\scT$ in the following sense.

\begin{theorem}\label{thm:models}
There exist admissible models $\hat Z^\eps = (\hat \Pi^\eps,\hat \Gamma^\eps)$ with $\eps\in[0,1]$ for $\scT$ such that for all $a\in S\cap T$, $\hat\Pi_x^\eps a=\bar\Pi_x^\eps a$ almost surely, and that one has the bounds
\begin{equs}
\E(\|\hat\Pi^\eps\|+\|\hat\Gamma^\eps\|)^p&\lesssim 1,
\\
\E(\|\hat\Pi^\eps-\hat\Pi^0\|+\|\hat\Gamma^\eps-\hat\Gamma^0\|)^p&\lesssim \eps^{p\theta}.
\end{equs}
\end{theorem}
\begin{proof}
Define the set $S'\subset S$ similarly to $S$, but with $\cC^{-N_0}$ replaced by $\{\partial^\ell\delta_c:\,c\in\C,|\ell|\leq N\}\subset \cC^{-N_0}$.
For $\tau\in\cW$,  $c\in\C^{[\tau]}$, $\ell\in (\N^{d_1})^{[\tau]}$ with $|\ell_i|\leq N$,
denote $a_{c,\ell}(\tau)=\big(\bigotimes_{i=1}^{[\tau]}\partial^{\ell_i}\delta_{c_i}\big)\otimes\tau$.
From \eqref{eq:model 1st bounds}, we have
\begin{equ}
\E\sup_{x,\lambda,\psi}
\lambda^{-p|\tau|}|
\big(\bar\Pi^\eps_x
(a_{c,\ell}(\tau)-a_{\bar c,\ell}(\tau) \big)
(\psi_x^\lambda)|^p
\lesssim|c-\bar c|^p,
\end{equ}
for any $c,\bar c\in\C^{[\tau]}$.
Choosing $p$ large enough, by Kolmogorov's continuity theorem one has a continuous modification $(\hat \Pi_xa_{c,\ell}(\tau))_{c\in\C^{[\tau]}}$ 
such that the admissibility condition is satisfied almost surely,
and that one has the bound
\begin{equ}
\E\sup_{x,\lambda,\psi}
\sup_{c\in\C^{[\tau]}}
\lambda^{-p|\tau|}|
(\hat\Pi^\eps_x
a_{c,\ell}(\tau))
(\psi_x^\lambda)|^p
\lesssim 1.
\end{equ}
Note that a generic element of $S\cap T$ is of the form $a=\big(\bigotimes_{i=1}^{[\tau]}\zeta_i\big)\otimes\tau$, with $\zeta_i\in \cC^{-N}=\CB$. Hence on $S\cap T$ we can define
\begin{equ}
\hat\Pi_x^\eps a:=\big(\textstyle{\bigotimes_{i=1}^{[\tau]}}\zeta_i\big)(c\mapsto\hat\Pi_x^\eps a_{c,0}(\tau)),
\end{equ}
and extending these maps to all of $T$ by linearity and continuity, we get maps $\hat \Pi_x^\eps$ that are admissible
and that satisfy 
\begin{equ}
\E\sup_{x,\lambda,\psi}
\sup_{0\neq a\in T}|a|^{-p}
\lambda^{-p|\tau|}|
(\hat\Pi^\eps_x
a)
(\psi_x^\lambda)|^p
\lesssim 1.
\end{equ}
The corresponding bounds on the differences $\hat\Pi_x^\eps-\hat\Pi_x^0$ is obtained similarly, so it remains to treat the maps $\hat\Gamma^\eps$.

We proceed inductively. 
The definition of, and the appropriate bounds on, $\hat\Gamma_{xy}^\eps \tau$ if $\tau=\Xi$ or $X^k$, are trivial. 
Given $\hat\Gamma_{xy}^\eps(\zeta\otimes\tau)$ and $\hat\Gamma_{xy}^\eps(\bar\zeta\otimes\bar\tau)$ with the right bounds,
we set
\begin{equ}
\hat\Gamma_{xy}\big((\zeta\otimes\bar\zeta)\otimes\tau\bar\tau\big)
=\big(\hat\Gamma_{xy}^\eps\big(\zeta\otimes\tau\big)\big)\big(\hat\Gamma_{xy}^\eps\big(\bar\zeta\otimes\bar\tau\big)\big),
\end{equ}
which one can bound by
\begin{equs}
\|\hat\Gamma_{xy}\big((\zeta\otimes\bar\zeta)\otimes\tau\bar\tau\big)\|_m
&\lesssim\sum_{m_1+m_2=m}
\|x-y\|_\frs^{|\tau|-m_1}|\zeta|_{\CB_{[\tau]}}
\|x-y\|_\frs^{|\bar\tau|-m_2}|\zeta|_{\CB_{[\bar \tau]}}
\\
&\lesssim\|x-y\|_\frs^{|\tau\bar\tau|-m}|\zeta\otimes\bar\zeta|_{\CB_{[\tau\bar\tau]}},
\end{equs}
where  we used our assumption on the spaces $\CB_k$ to obtain the second line.
Given $\hat\Gamma_{xy}^\eps\zeta\otimes\tau$, we also set $\hat\Gamma_{xy}^{\eps}(\zeta\otimes\scD^i\tau) =\scD_i\hat\Gamma_{xy}^{\eps}(\zeta\otimes\tau)$, for which the correct bounds follow automatically.

The only step to finish the induction is thus to define and bound $\hat\Gamma^\eps_{xy}\zeta\otimes\cI\tau$,
provided $\hat\Gamma^\eps_{xy}\zeta'\otimes\tau$ are known.
This is done as in \cite[Thm~5.14]{H0}: for $\zeta_1\in \CB$, $a=\zeta'\otimes\tau$, we set
\begin{equ}\label{eq:Gamma}
\hat\Gamma^\eps_{xy}\cI^{\zeta_1}a=\cI^{\zeta_1}a+\cI^{\zeta_1}(\hat\Gamma^\eps_{xy}a-a)+\big(\cJ^{\zeta_1}(x)\hat\Gamma_{xy}^\eps a-\hat\Gamma_{xy}^\eps\cJ^{\zeta_1}(y)a\big)\,.
\end{equ}
The first term on the right-hand side is harmless.
Bounding the second one is immediate:
\begin{equ}
\|\cI^{\zeta_1}(\hat\Gamma^\eps_{xy}a-a)\|_m\lesssim|\zeta_1|_{\CB}\|\hat\Gamma^\eps_{xy}a-a\|_{m-\beta}\lesssim|\zeta_1|_{\CB}|\zeta'|_{\CB_{[\tau]}}\|x-y\|^{|\tau|+\beta-m}
\end{equ}
thanks to the assumed bound on $\hat\Gamma^\eps_{xy}\zeta'\otimes\tau$. Using again the assumptions on the spaces $\CB_k$, this is precisely the required bound.
To bound the third term on the right-hand side of \eqref{eq:Gamma}, it suffices to recall \cite[Lem~5.21]{H0},
with the kernel $K$ therein replaced by $K^{\zeta_1}$.
Having the required bounds on elements of the form $\hat\Gamma_{xy}^\eps(\zeta_1\otimes\zeta'\otimes\cI\tau)$,
one can extend $\hat\Gamma_{xy}^\eps$ to all $a\in T_{\cI\tau}$ once again via linearity and continuity.
It is straightforward to check that the above defined maps $\hat\Gamma^\eps_{xy}$ do indeed belong to $G$, and hence the proof is finished.
\end{proof}

\begin{remark}\label{rem:renormalisation}
Let us comment briefly on the renormalisation procedure implicit in the construction
\eqref{e:constr}.
In the standard situation considered in \cite{BHZ,CH}, the BPHZ renormalisation
procedure assigns to each symbol $\tau \in \cW$ with $\tau \neq \bone$ and 
$|\tau| \le 0$ a constant $C^\eps_\tau$. (In the notation of \cite[Eq.~6.22]{BHZ}, one 
has $C^\eps_\tau = g_-(\PPi^\eps)(\iota_\circ \tau) = \E (\PPi^\eps\tau)(0)$,
with $\PPi^\eps$ the canonical lift of the mollified noises.) 
This choice then allows to define
a renormalised model by \cite[Thms~6.17,~6.27]{BHZ} which was shown in \cite[Thms~2.14,~2.30]{CH} to enjoy very strong
stability properties. 

The construction given above is essentially the same, but now each symbol $\tau$
determines a smooth \textit{function} $C^\eps_\tau \colon \C^{[\tau]} \to \R$,
where 
\begin{equ}[e:formulaC]
C^\eps_\tau(c) = \E \PPi^\eps\big(\big(\textstyle{\bigotimes_{i=1}^{[\tau]}}\delta_{c_i}\big)\otimes\tau\big)(0)\;.
\end{equ}
The construction of the renormalised model is then the same as in \cite{BHZ}.
\end{remark}

\section{Lifting the operator $I$}\label{sec:4}

We continue within the setting of the previous section.
Given now that we have abstract 
integration operators $\cI^\zeta$ on $T$ that
that can in principle be used as in \cite[Sec.~4]{H0} to build the operation of convolution
with any of the $K^{\zeta}$,
we are also able to construct the abstract counterpart of the operators $I_k^{(\ell)}$,
acting on suitable spaces of modelled distributions.

From now on we assume $d>1$ and the first coordinate will be viewed as time.
We work with $\cD^{\gamma,\eta}_P$ spaces defined as in \cite[Sec~6]{H0}, with $P=\{(0,x):\,x\in\T^{d-1}\}$.
It will be clear that apart from notational inconvenience there is no fundamental obstacle to obtaining analogous 
results for more complicated weighted spaces, like for example those considered in \cite{GH17} that are suitable 
for solving initial-boundary value problems.

Given the setup of the previous section and an admissible model $(\Pi,\Gamma)$, one can define the maps $\cK^\zeta_m$ by replacing $\cI$ and $K$ in
\cite[Eq~5.15]{H0} by $\cI^{(m),\zeta}$ and $D_mK^\zeta$, respectively,
provided $|m|_\frs<\beta$. 
As before, we denote $\cK^{c;\ell}_m:=\cK^{\partial^\ell\delta_c}_m$,
and for $m=0$ the lower index is omitted.

We now define the lift of $I$ by a sort of 'higher order freezing of coefficients'
where, around a given fixed point $z_0$, we don't simply describe $I(b,f)$ by $\cK^{b(z_0);0} f$,
but also use higher order information about $b$. 
Set, with $\bar b=\scal{b,\bone}$ and $\hat b=b-\bar b$,
\begin{equ}\label{frI}
\frI^{(m)}_k(b,f)(z):=\sum_{|\ell|\le N'}\frac{(\hat b(z))^{ \ell}}{\ell!}
(\cK^{\bar b(z);k+\ell}_m f)(z)\;,
\end{equ}
where $N'$ is a sufficiently large integer. (How large exactly will be specified in the statement
of Theorem~\ref{thm:SchauderI} below. Since the exact value of $N'$ does not make much of a difference for 
our purpose, we do not explicitly keep track of it in our notations.) 
In the following we treat only $\frI:=\frI^{(0)}_0$.
The Schauder estimate for $\frI^{(m)}_k$ can then be formally obtained by
changing the family of kernels $(K^{(c)})_{c\in\C}$ to $(\partial_c^k\partial_x^m K^{(c)})_{c\in\C}$, as well as $\beta$ to $\beta-|m|_\frs$, and apply the Schauder estimate for the map $\frI$ built from this family.

Note that the definition \eqref{frI} is very reminiscent of how one composes modelled distributions with smooth functions $F$, see \cite[Sec~4.2]{H0}.
To justify this analogy, one needs a substitute for the Taylor expansion of $F$, which is precisely the content of Corollary \ref{cor:good} below.
Thanks to this (of course not coincidental, see Remark~\ref{rem:composition} below) similarity, the Schauder estimates for $\frI$ 
will follow immediately from the one for `constant coefficients' (Theorem \ref{thm:Schauder} below),
and a straightforward adaptation of the proof of \cite[Prop~6.13]{H0}.

Recall that we previously fixed $\alpha\in\R$, $\beta>0$. Fix further some $\gamma_1,\gamma_2,\bar\gamma>0$ and $\eta_1$,
$\eta_2$,
and $\bar \eta$ such that
\begin{equ}
\bar\gamma\leq\big(\gamma_1+(\alpha+\beta)\wedge0\big)\wedge\big(\gamma_2+\beta\big),
\quad
\alpha>\eta_2>-\frs_1,
\quad \eta_2+\beta>\bar\eta,
\quad \eta_1\geq\bar\eta\vee0.
\end{equ}

\begin{remark}
Note that if $\alpha+\beta \ge 0$ and $\eta_1 \ge 0$, then 
one can simply choose  $\bar\gamma=\gamma_1=\gamma_2+\beta$ and $\bar\eta=\eta_1 < \eta_2 + \beta< \alpha+\beta$.
\end{remark}

We assume henceforth that the kernels $K^{(c)}$ are non-anticipative, namely that they 
vanish for negative times. One then has the following, see \cite[Thm~7.1]{H0}.
\begin{theorem}\label{thm:Schauder}
With $\kappa=(\eta_2+\beta-\bar\eta
)/\frs_1$, for any $\zeta\in\cC^{-N}$, $f\in\cD^{\gamma_2,\eta_2}_P$, and any $t\in(0,1]$, one has
\begin{equ}\label{eq:Schauder}
\vn{\cK^\zeta\bR^+f}_{\gamma_2+\beta,\bar\eta;t}\lesssim t^\kappa|\zeta|_{\cC^{-N}}\vn{f}_{\gamma_2,\eta_2;t}.
\end{equ}
\end{theorem}

\begin{corollary}\label{cor:good}
Let $f\in\cD^{\gamma_2,\eta_2}_P(V)$. Then for $c,\bar c\in \C$, $\ell\in\mathbb{N}^{d_1}$, $m\geq 0$ with $m+|\ell|+d_1+1\leq N$, and any $t\in(0,1]$, one has
\begin{equ}\label{eq: K derivative}
\Big\|\cK^{c;\ell} \bR^+f
-\sum_{|k|\leq m}\frac{(c-\bar c)^k}{k!}\cK^{\bar c;\ell+k}\bR^+f\Big\|_{\gamma_2+\beta,\bar\eta;t}
\lesssim|c-\bar c|^{m+1}t^{\kappa}\vn{f}_{\gamma_2,\eta_2;t}.
\end{equ}
\end{corollary}
\begin{proof}
Simply apply Theorem~\ref{thm:Schauder} with 
$
\zeta=\partial^\ell\delta_c - \sum_{|k|\leq m}\frac{(c-\bar c)^k}{k!}\partial^{k+\ell}\delta_{\bar c}
$.
\end{proof}

\begin{theorem}\label{thm:SchauderI}
Assume the above setting and suppose that $b\in\cD^{\gamma_1,\eta_1}_P(V)$ is $\C$-valued, $f\in\cD^{\gamma_2,\eta_2}_P$, where $V$ is a function-like sector with lowest nonzero homogeneity $\alpha_1$ and $N'\alpha_1>\gamma_2+\beta$.
Then $\frI(b,f)\in\cD^{\bar\gamma,\bar\eta}_P$.

If $\tilde b\in\cD^{\gamma_1,\eta_1}_P(V,\tilde\Gamma)$, $\tilde f\in\cD^{\gamma_2,\eta_2}_P(T,\tilde\Gamma)$
with another admissible model $(\tilde\Pi,\tilde\Gamma)$, then one has the following bound, for any $t\in(0,1]$,
\begin{equ}
\vn{\frI(b,\bR^+ f)\,;\frI(\tilde b,\bR^+ \tilde f)}_{\bar\gamma,\bar\eta;t}\lesssim t^{\kappa}
\big(\vn{b\,;\tilde b}_{\gamma_1,\eta_1;t}+\vn{f\,;\tilde f}_{\gamma_2,\eta_2;t}+\|(\Pi,\Gamma)-(\tilde\Pi,\tilde\Gamma)\|_{\bar\gamma;2}\big).
\end{equ}
Moreover, if $\alpha+\beta>0$, then the identity
\begin{equ}\label{eq:integration compatiible}
\cR\frI(b,f)=I(\cR b,\cR f),
\end{equ}
holds, where $I$ is defined as in \eqref{eq: I}.
\end{theorem}
\begin{remark}
Note that if $\alpha_1+\alpha+\beta<0$, then the equality \eqref{eq:integration compatiible} fails to hold in general even for canonical models built from a smooth noise.
\end{remark}

\begin{proof}
Denoting $F^{(\ell)}(c,z)=(\cK^{c;\ell}\bR^+f)(z)$, since $\bar\gamma\leq\gamma_2+\beta$,
one has that
$F^{(\ell)}(c,\cdot)$ is a modelled distribution with its $\vn{\cdot}_{\bar\gamma,\bar\eta; t}$ norm bounded by $t^\kappa$, and by Corollary \ref{cor:good} the map $c \mapsto F^{(0)}(c,\cdot)$ is smooth (in the usual sense) into $\cD_P^{\bar\gamma,\bar\eta}$
with its derivatives given precisely by the $F^{(\ell)}$.

It then follows from the multiplicativity of the action of the structure group that
\begin{equs}
\sum&\frac{(\hat b(x))^{ \ell}}{\ell!}
F^{(\ell)}(\bar b(x), x)
-\Gamma_{xy}\Big(\sum\frac{(\hat b(y))^{ \ell}}{\ell!}
F^{(\ell)}(\bar b(y), y)\Big)
\\&=
\Big(\sum\frac{(\hat b(x))^{ \ell}}{\ell!}
F^{(\ell)}(\bar b(x), x)
-
\sum\frac{(\Gamma_{xy}\hat b(y))^{ \ell}}{\ell!}
F^{(\ell)}(\bar b(y), x)\Big)
\\
&\quad+\sum\frac{(\Gamma_{xy}\hat b(y))^{ \ell}}{\ell!}
\big(F^{(\ell)}(\bar b(y), x)-\Gamma_{xy}F^{(\ell)}(\bar b(y),y)\big)=:A_1+A_2
\end{equs}
The term $A_1$ can be bounded precisely as in \cite[Prop~6.13]{H0}, with the only minor difference
that the smooth function $F^{(\ell)}(\cdot,x)$ that $b$ is substituted into, takes values in $T$ instead of $\bR$.
One then gets
\begin{equ}
\|A_1\|_m\lesssim t^\kappa\sum_{m_1+m_2=m}\|x-y\|_\frs^{\gamma_1-m_1}|x|_P^{\eta_1-\gamma_1}|x|_P^{(\bar\eta-m_2)\wedge0},
\end{equ}
where in the above sum $m_2$ runs over homogeneities of $\cI\cW+\bar T$, in particular its smallest value is $(\alpha+\beta)\wedge 0\geq\bar\gamma-\gamma_1$.
Therefore,
\begin{equs}
\|A_1\|_m&\lesssim t^\kappa\sum_{m_1+m_2=m}\|x-y\|_\frs^{\bar\gamma-m}\|x-y\|_\frs^{m_2+\gamma_1-\bar\gamma}|x|_P^{\eta_1-\gamma_1}|x|_P^{(\bar\eta-m_2)\wedge0}
\\
&\lesssim t^\kappa\sum_{m_1+m_2=m}\|x-y\|_\frs^{\bar\gamma-m}|x|_P^{\eta_1-\bar\gamma}|x|_P^{\bar\eta\wedge m_2}\lesssim\|x-y\|_\frs^{\bar\gamma-m}|x|_P^{\bar\eta-\bar\gamma},
\end{equs}
where in the last step we used $\eta_1\geq\bar\eta\vee0$ and $\alpha+\beta>\bar\eta$.
One the other hand, 
\begin{equ}
\|A_2\|_m\lesssim t^\kappa\sum_{m_1+m_2=m}\|x-y\|_\frs^{-m_1}\|x-y\|_\frs^{\bar\gamma-m_2}|x|_P^{\bar\eta-\bar\gamma}\lesssim t^\kappa\|x-y\|_\frs^{\bar\gamma-m}|x|_P^{\bar\eta-\bar\gamma}
\end{equ}
as required.

For a fixed model, bounding $\vn{\frI(b,\bR^+f);\,\frI(b,\bR^+\tilde f)}_{\bar\gamma,\bar\eta;t}$ is immediate
from the above thanks to the linearity of $\frI$ in the second argument.
To bound $\vn{\frI(b,\bR^+f);\,\frI(\tilde b,\bR^+f)}_{\bar\gamma,\bar\eta;t}$, one can write, as in the proof of \cite[Thm~4.16]{H0}, with $b' =b-\tilde b$,
\begin{equs}\label{eq:int diff}
\big(\frI(b,\bR^+f)&-\frI(\tilde b,\bR^+f)\big)(x)
\\&=\sum_{\ell,i}\int_0^1 (b')_i(x)
\frac{(\hat{\tilde b}(x)+\theta \hat b'(x))^{ \ell}}{\ell!}(\cK^{\bar{\tilde b}(x)+\theta\bar b'(x);\ell+e_i}\bR^+ f)(x)\,d\theta,
\end{equs}
where the sum over $i$ runs over $1,\ldots,d_1$, and $e_i$ is the $i$-th unit vector in $\R^{d_1}$.
Now one can repeat the preceding calculation, with `gaining' a factor $\vn{b'}_{\gamma_1,\eta_1;\,t}$ at each step.

Finally, to bound $\vn{\frI(b,\bR^+f);\,\frI(\tilde b,\bR^+\tilde f)}_{\bar\gamma,\bar\eta;t}$ for two different models, one can employ the trick in \cite[Prop~3.11]{HP15}.
\end{proof}

\begin{remark}\label{rem:composition}
The same argument actually shows that if $c \mapsto F(c,\cdot)$ is a smooth function from $\C$ to 
$\CD_P^{\bar\gamma,\bar\eta}$ and $b = \bar b\bone + \hat b$ is as in the statement, then the function $G$ given by
\begin{equ}
G(z) = \sum_{|\ell|\le N}\frac{(\hat b(z))^{ \ell}}{\ell!}
F^{(\ell)}(\bar b(z), z)
\end{equ}
belongs to $\CD_P^{\bar \gamma,\bar \eta}$.
This statement then has both the first part of Theorem~\ref{thm:SchauderI} and \cite[Thm~4.16]{H0} 
as corollaries.
\end{remark}

To formulate the abstract counterpart of \eqref{eq: mod nl3}, it remains to lift the operators $\hat I^{(\ell)}_k$.
Using the notation
\begin{equ}
(K^\zeta u_0)(z)=\int K_t^\zeta(x-x') u_0(x')\,dx',
\end{equ}
and identifying this function with its lift via its Taylor expansion,
we define, similarly to $\frI$, 
\begin{equ}\label{eq:hat frI}
\hat\frI^{(m)}_k(b,u_0)(z):=\sum_{|\ell|\le N'}\frac{(\hat b(z))^{ \ell}}{\ell!}
(K^{\bar b(z);k+\ell} D_mu_0)(z)\;.
\end{equ} 
Further to the preceding we fix a non-integer exponent $1>\eta_0>\bar\eta$.
This time the `constant coefficient' result we rely on is the following variant of \cite[Lem~7.5]{H0}.
\begin{lemma}\label{lem:initial}
Assume $\beta=\frs_1$. Let $u_0\in\cC^{\eta_0}$ and $\zeta\in\cC^{-N}$. Then $K^\zeta u_0\in\cD^{\gamma,\eta_0}_P$
for any $\gamma\geq\eta_0\vee0$, and 
\begin{equ}
\|K^\zeta u_0\|_{\gamma,\eta_0}\lesssim |\zeta|_{\cC^{-N}}|u_0|_{\cC^{\eta_0}}.
\end{equ}
\end{lemma}
The behaviour of $\hat\frI$ is then given by the following lemma.
\begin{lemma}
Assume $\beta=\frs_1$. Let $V$, $N'$, $b$ and $\tilde b$ be as in Theorem \ref{thm:SchauderI}, and let $u_0,\tilde u_0\in\cC^{\eta_0}$.
Then $\hat\frI(b,u_0)\in\cD^{\bar\gamma,\bar\eta}_P$ and with $\bar\kappa=(\eta_0-\bar\eta)/\frs_1$ one has the bounds, for any $t\in(0,1]$,
\begin{equs}
\vn{\hat\frI(b,u_0)}_{\bar\gamma,\bar\eta;t}&\lesssim
\vn{b}_{\gamma_1,\eta_1;t}|u_0|_{\cC^{\eta_0}},
\\
\vn{\hat\frI(b,u_0)\,;\hat\frI(\tilde b,\tilde u_0)}_{\bar\gamma,\bar\eta;t}&\lesssim t^{\bar\kappa}
\big(\vn{b\,;\tilde b}_{\gamma_1,\eta_1;t}+\|(\Pi,\Gamma)-(\tilde\Pi,\tilde\Gamma)\|_{\bar\gamma;2}\big)+|u_0-\tilde u_0|_{\cC^{\eta_0}}.
\end{equs}
Moreover, the following identity holds
\begin{equ}
\cR\hat\frI(b,u_0)=\hat I(\cR b,u_0).
\end{equ}
\end{lemma}
\begin{proof}
The proof goes precisely as that of Theorem \ref{thm:SchauderI}, with the only slight difference that the `constant coefficient' estimate seemingly does not help in obtaining a positive power of $t$.
Note however that whenever $\eta_0>0$, for any $c\in\C$ and nonzero multiindex $\ell$, $\scal{K^{c;\ell} u_0,\bone}$ vanishes at the initial time.
In particular, whenever $\eta_0<1$, all components of $K^{c;\ell} u_0$ lower than $\eta_0$ vanish at the initial time, and hence (see \cite[Lem~6.5]{H0}) one gets the estimate
\begin{equ}
\|K^{c;\ell} u_0\|_{\gamma,\bar\eta}\lesssim t^{\bar\kappa}|u_0|_{\cC^{\eta_0}}.
\end{equ}
It remains to notice that in the calculation for $\vn{\hat\frI(b,u_0)\,;\hat\frI(\tilde b, u_0)}_{\bar\gamma,\bar\eta;t}$ analogous to \eqref{eq:int diff}
we only ever encounter instances of $K$ with nonzero derivatives with respect to the parameter $c$,
hence the claimed factor $t^{\bar\kappa}$ in the lemma is indeed obtained.
\end{proof}

\section{A concrete example}\label{sec:example}

At this point, we have a completely automatic solution theory: given a quasilinear equation like
\eqref{eq: 00}, its solution is \emph{defined} as $\cR U$, where $U$ is obtained from the
system of abstract equations
\begin{equs}[eq:system]
U &= \frI(a(U), \hat \cF) + \hat\frI(a(U),u_0)\;,\\
\hat \cF&= (1- V_3 a'(U) )F(U,\Xi) + 2 V_1 a(U)a'(U)\scD U + V_2 a(U) (a'(U))^2 (\scD U)^2\\
&\quad + V_3 a(U) a''(U) (\scD U)^2
\;,\\
V_1 &= \frI_1'(a(U), \hat \cF) + \hat \frI_1'(a(U),u_0)\;,\\
V_2 &= \frI_2(a(U), \hat \cF) + \hat \frI_2(a(U),u_0)\;,\\
V_3 &= \frI_1(a(U), \hat \cF) + \hat \frI_1(a(U),u_0)\;.
\end{equs}
If $F$ was a subcritical nonlinearity to begin with and $\alpha>-2$, then the above system is again subcritical,
so one can use the construction of Section~\ref{sec:structure} to
build the corresponding regularity structure.
Provided it satisfies Assumption~\ref{asn:supreg},
one can use \cite{CH} in the form of Theorem~\ref{thm:models}
to obtain the corresponding BPHZ model.
The local well-posedness of \eqref{eq:system} is then a standard consequence of the results of 
Section~\ref{sec:4} above just as in \cite[Sec~6]{H0}.

However, as mentioned in the introduction, at this point it is not automatic
to see what counterterms appear (or whether they are even local in the solution)
in the equation solved by $\cR^\eps U^\eps$, where $U^\eps$ is obtained from solving \eqref{eq:system} with 
a renormalised smooth model.
Below we carry out the computation of these terms in the setting of the example \eqref{eq:example}.
An interesting outcome of these calculations is that if we consider the BPHZ renormalisation of our
model, then it may happen in general that non-local counterterms appear. However, as we will see,
it is possible to choose the renormalisation procedure in such a way that these non-local terms cancel
out, thus leading to the stated result.

\begin{proof}[Proof of Theorem \ref{thm:example}]
Our abstract equation reads as \eqref{eq:system}, with $F(U,\Xi)=F_0(U)(\scD U)^2+F_1(U)\Xi$.
The regularity structure is built as discussed above, where we declare the homogeneity of $\Xi$
to be $-3/2+\kappa$ for some $\kappa\in(0,(\nu\wedge\bar \nu)/2)$.
As for the models, we take a slight modification of the associated BPHZ models $(\hat \Pi^\eps,\hat\Gamma^\eps)$ obtained from Theorem \ref{thm:models}.

Recall first from Remark~\ref{rem:renormalisation} that the BPHZ renormalisation procedure is
parametrised by functions $C_\tau^\eps$ given by \eqref{e:formulaC} for $\tau \in \cW_- := \{\tau \in \cW \setminus \{\bone\}\,:\, |\tau|\le 0\}$. As a consequence of the fact that we choose $|\Xi| > -{3\over 2}$,
one can verify that all $\tau \in \cW_-$ satisfy $\db{\tau} \le 3$.
Since we furthermore assumed that the driving noise $\xi$ is centred Gaussian, the functions 
$C_\tau^\eps$ vanish identically for all $\tau$ with $\db{\tau}$ odd, so that only symbols
with $\db{\tau} = 2$ contribute to the renormalisation.

Using the graphical notation from \cite{HP15,String} (circles represent $\Xi$, plain lines represent
$\cI$ and bold red lines represent $\cI'$), the only two such symbols 
are given by \<Xi2> and \<IXi^2>. The corresponding renormalisation functions
are given by
\begin{equ}[e:defC]
C_{\<Xi2s>}^\eps(c)=\int K^{(c)}(z)\scC^\eps( z)\,dz,\quad C_{\<IXi^2s>}^{\eps}(c,\bar c)=\int \partial_xK^{(c)}(z)\partial_xK^{(\bar c)}(\bar z)\scC^\eps(\bar z-z)\,dz\,d\bar z,
\end{equ}
where $\scC^\eps=\scC\ast\rho^\eps\ast\rho^\eps$. 
In this particular case, this allows us as in \cite{BHZ} to define linear maps $M^\eps\colon T \to T$ such
that on $T_{\le 0}$ the BPHZ renormalised model $(\hat \Pi^\eps,\hat \Gamma^\eps)$  satisfies the identity
\begin{equ}[e:simple]
\hat \Pi^\eps_z \tau = \Pi^\eps_z M^\eps \tau\;,
\end{equ}
where $\Pi^\eps$ is the canonical lift of $\xi^\eps$.
(The fact that \eqref{e:simple} holds is no longer the case when $\kappa \le 0$!)
One has for example
\begin{equ}[e:renorm]
M^\eps (\zeta \otimes \<Xi2>) = \zeta \otimes \<Xi2> - \zeta(C_{\<Xi2s>}^\eps)\bone\;,\quad 
M^\eps (\zeta \otimes\eta\otimes \nu\otimes  \<IXi21>) = \zeta \otimes\eta\otimes \nu\otimes  \<IXi21> - (\zeta\otimes \eta)(C_{\<IXi^2s>}^\eps)\, \nu \otimes \<IXi>\;.
\end{equ}
At this point we note that 
if $K^{(c)}$ were exactly equal to the heat kernel instead of a compactly supported truncation,
then one would have the exact identity
\begin{equ}[e:defC1]
C_{\<Xi2s>}^\eps(c) = c C_{\<IXi^2s>}^{\eps}(c,c)\;.
\end{equ}
It turns out that this identity is crucial in order to obtain the cancellations necessary to obtain local
counterterms. We therefore define a model $(\tilde \Pi^\eps,\tilde \Gamma^\eps)$
just like the BPHZ model, but with $C_{\<Xi2s>}^\eps$ \textit{defined} by \eqref{e:defC1} instead of \eqref{e:defC}.
Since the difference between these two different definitions of $C_{\<Xi2s>}^\eps$ converges to a finite smooth
function as $\eps \to 0$, the convergence of the BPHZ model as $\eps \to 0$ also implies the convergence
of the model $(\tilde \Pi^\eps,\tilde \Gamma^\eps)$.

Note also that (modulo changing the order of some factors: recall that $\<IXi^22>\neq \<IXi^22b>$ in our
setting, but this distinction is essentially irrelevant since we always consider models such that for example
$\Pi_x (\zeta \otimes \eta \otimes \nu \otimes \<IXi^22>) = \Pi_x (\nu \otimes \zeta \otimes \eta \otimes \<IXi^22b>)$,
so that we can identify such elements for all practical purposes), one has
\begin{equ}
\cW_- = \{\<Xi>, \<Xi>X_2, \<I'Xi>, \<Xi2>,  \<IXi^2>, \<IXi21>, \<XiIXi2>, \<Xi22>,\<IXi^22>, \<tree>,\<ladder>\}\;.
\end{equ}
Inspecting \eqref{e:renorm}, as well as the analogous expressions for all other symbols of negative degree, 
we conclude that 
one has 
\begin{equ}[e:Pirenorm]
(\tilde \Pi_z^\eps \tau)(z) = (\Pi_z^\eps \tau)(z) + \scal{\bone,(M^\eps - \id)\tau}
\end{equ}
for all $\tau \in T$,
where $\scal{\bone,\sigma}$ denotes the coefficient of $\bone$ in $\sigma$.
Furthermore, the second term in this expression  is non-vanishing only if $\tau$ contains a summand in 
$T_{\<Xi2s>}$ and / or in $T_{\<IXi^2s>}$.
This is because of \eqref{e:simple}, combined with the fact that $M^\eps$ only generates terms of 
strictly positive degrees for the remaining symbols in $\cW$.

%

We now have everything in place to derive the form of the renormalised equation.
Given the $(\tilde \Pi^\eps,\tilde\Gamma^\eps)$ for some fixed $\eps > 0$, one obtains a local 
solution of the system \eqref{eq:system} in 
\begin{equ}
\cD^{3/2+2\kappa,2\kappa}(W_0)\oplus\cD^{\kappa,-2+3\kappa}\oplus
\cD^{1/2+2\kappa,-1+2\kappa}(W_1)\oplus\cD^{1+2\kappa,2\kappa}(W_0)\oplus
\cD^{3/2+2\kappa,2\kappa}(W_0),
\end{equ}
where $W_0$ is the sector generated by the Taylor polynomials and elements of the form $\cI^\zeta\tau$, and $W_1=\scD W_0$.
As a consequence of \eqref{e:Pirenorm}, we conclude as for example in \cite[Sec.~9.3]{H0} that for $\eps>0$ 
the pair $(\cR^\eps U^\eps,\cR^\eps \hat \cF^\eps)$ solves an equation just like \eqref{e:system},
but with an additional term $\scal{\bone,(M^\eps-\id)\hat \cF^\eps}$ appearing on the right-hand side of \eqref{eq: mod nl3}.
Hence $\cR^\eps U^\eps$ solves an equation just like \eqref{eq: 00}, but with an additional term
\begin{equ}\label{eq:correction}
\cE := \frac{\scal{\bone,( M^\eps-\id)\hat \cF^\eps}}{1-a'(u)\scal{\bone,V_3^\eps}}
\end{equ}
appearing on the right-hand side.


It now remains to show that if $(U, \hat \cF)$ solve \eqref{eq:system}, then 
\eqref{eq:correction} coincides with a local functional of $u = \cR U = \scal{\bone, U}$.
Write furthermore $v_i= \cR V_i = \scal{\bone,V_i}$, as well as $q=1-v_3a'(u)$, where the $V_i$ are
as in \eqref{eq:system}.
Note that $q$ is the denominator in \eqref{eq:correction}, and that this is \textit{not} a local
functional of $u$, so that we should aim for a factor $q$ to appear in the numerator as well.
To ease notation, we henceforth also omit the lower indices in $\delta_{a(u)}$ and $\delta'_{a(u)}$.
Since all symbols appearing in the expansion of the solution
are of the form $\zeta\otimes\tau$, where $\zeta$ is a tensor product of either $\delta_{a(u)}$
or $\delta'_{a(u)}$, this will hopefully not cause any confusion.

To calculate the numerator in \eqref{eq:correction}, it follows from the above discussion
that we only need to know the components of $\hat\cF$ in $T_{\<Xi2s>}$ and in $T_{\<IXi^2s>}$.
For this, note first that one has
\begin{equ}[e:RHSF]
\hat \cF = q F_1(u)\, \<Xi> + (\ldots)\;,
\end{equ}
where all terms included in $(\ldots)$ are of strictly higher degree than that of $\Xi$.
Combining \eqref{eq:system} with the definitions of $\frI$ and $\frI_1$, we then see that 
\begin{equ}
U = u\,\bone + u_{\<IXis>} \otimes \<IXi> + \tilde U\;,
\end{equ}
where $\tilde U$ takes values only in spaces $T_\tau$ with $\tau\neq \<IXi>$ of the form
$\tau = \prod_{i}\cI^{\zeta_i}(\sigma_i)$. Furthermore, by \eqref{e:RHSF} and 
the definition of $V_3$, the distribution $u_{\<IXis>}$ is given by
\begin{equ}
u_{\<IXis>} = q F_1(u)\delta + a'(u)v_3 u_{\<IXis>}\qquad \Rightarrow\qquad
u_{\<IXis>} = F_1(u)\delta\;,
\end{equ}
so that in particular
\begin{equ}
a(U) = a(u)\,\bone + (a'F_1)(u)\delta \otimes \<IXi> + (\ldots)\;.
\end{equ}
Combining this with \eqref{e:RHSF} and the expressions for $V_i$, we conclude similarly that
\begin{equs}
V_1 &= qF_1(u)\delta' \otimes \<I'Xi> +  v_1\,\bone + (\ldots)\;,\qquad
V_2 = v_2\,\bone  + (\ldots)\;,\\
V_3 &= v_3\,\bone + \bigl(qF_1(u)\delta' + v_2 (a'F_1)(u)\delta\bigr) \otimes \<IXi> + (\ldots)\;,
\end{equs}
where the terms denoted by $(\ldots)$ are of higher degree.
Combining all of this with the expression for $\hat \cF$ in \eqref{eq:system}, we finally obtain 
the next order in the development of $\hat \cF$, namely
\begin{equs}
\hat \cF &= q F_1(u)\, \<Xi> + 
\bigl(q(F_1'F_1)(u)-v_3(a''F_1^2)(u)-v_2((a')^2F_1^2)(u)\bigr)\,\delta \otimes \<Xi2> \\
&\quad -q(a'F_1^2)(u) \delta' \otimes \<Xi2> + 2q(aa'F_1^2)(u)\,\delta'\otimes \delta \otimes \<IXi^2> \\
&\quad + \bigl(q(F_1^2F_0)(u)+v_2(a(a')^2F_1^2)(u)+v_3(aa''F_1^2)(u)\bigr) \delta\otimes \delta \otimes \<IXi^2> + (\ldots)\;.
\end{equs}
Combining this with the definition of $M^\eps$, we conclude that the counterterm \eqref{eq:correction} is given by
\begin{equs}
\cE &= -{1\over q} \bigl(q(F_1'F_1)(u)-v_3(a''F_1^2)(u)-v_2((a')^2F_1^2)(u)\bigr)C_{\<Xi2s>}^\eps(a(u)) \\
&\quad +  (a'F_1^2)(u) (\d C_{\<Xi2s>}^\eps)(a(u)) - 2(aa'F_1^2)(u) (\d_1 C_{\<IXi^2s>}^\eps)(a(u),a(u)) \\
&\quad -{1\over q}\bigl(q(F_1^2F_0)(u)+v_2(a(a')^2F_1^2)(u)+v_3(aa''F_1^2)(u)\bigr)C_{\<IXi^2s>}^\eps(a(u),a(u))\;.
\end{equs}
At this point we note that by \eqref{e:defC1} and the fact that $C_{\<IXi^2s>}^\eps$ is symmetric in its two
arguments, one has the identity
\begin{equ}
(\d C_{\<Xi2s>}^\eps)(c) = C_{\<IXi^2s>}^\eps(c,c) + 2c(\d_1 C_{\<IXi^2s>}^\eps)(c,c)\;.
\end{equ}
Inserting this into the above equation and noting that the terms proportional to $v_2$ and $v_3$ cancel
out exactly thanks again to \eqref{e:defC1}, we conclude that
\begin{equ}
\cE = - C_{\<Xi2s>}^\eps(a(u)) \bigl((aF_1'F_1)(u) + (F_1^2F_0)(u) - (a'F_1^2)(u)\bigr)/a(u)\;,
\end{equ}
which is precisely as in \eqref{eq:example reno}.
\end{proof}

\begin{remark}
The expression \eqref{e:defC} also gives some information about the behaviour of $C_{\<Xi2s>}^\eps$
in the case where $\scC$ is self-similar
on small scales, i.e. $\scC(\lambda^2 t,\lambda x)=\lambda^{-3+\nu}\scC(t,x)$ for all $\lambda\in(0,1]$ and $|t|^{1/2}+|x|\leq r$, for some $r>0$.
Indeed, one can then write
\begin{equs}
C_{\<Xi2s>}^\eps(c) &=\int K^{(c)}(z)\scC^\eps( z)\,dz\approx\eps^{-3+\nu}\int K^{(c)}(z)\scC^1( \eps z)\,dz
\\
&=\eps^{\nu}\int K^{(c)}(\eps^{-1}z)\scC^1( z)\,dz\approx\eps^{-1+\nu}\int K^{(c)}(z)\scC^1( z)\,dz,
\end{equs}
where $\approx$ means that the difference of the two sides converge as $\eps\rightarrow0$ to a smooth function of $c$.
Hence, modulo changing again the renormalisation constants by a finite quantity, one can use in this
case a counterterm of the form $\eps^{\nu-1} A(u)$ for some explicit function $A$ of the solution $u$. 
\end{remark}

\bibliography{Quasi}{}

\begin{thebibliography}{OSSW17}
\expandafter\ifx\csname url\endcsname\relax
  \def\url#1{\texttt{#1}}\fi
\expandafter\ifx\csname urlprefix\endcsname\relax\def\urlprefix{URL }\fi
\expandafter\ifx\csname href\endcsname\relax
  \def\href#1#2{#2}\fi
\expandafter\ifx\csname burlalt\endcsname\relax
  \def\burlalt#1#2{\href{#2}{\texttt{#1}}}\fi

\bibitem[BCCH17]{BCCH}
\textsc{Y.~{Bruned}}, \textsc{A.~{Chandra}}, \textsc{I.~{Chevyrev}}, and
  \textsc{M.~{Hairer}}.
\newblock {Renormalizing SPDEs in regularity structures}.
\newblock \emph{ArXiv e-prints} (2017).
\newblock \burlalt{arXiv:1711.10239}{http://arxiv.org/abs/1711.10239}.

\bibitem[BDH16]{BDH}
\textsc{I.~Bailleul}, \textsc{A.~Debussche}, and \textsc{M.~{Hofmanov\'a}}.
\newblock {Quasilinear generalized parabolic Anderson model equation}.
\newblock \emph{ArXiv e-prints} (2016).
\newblock \burlalt{arXiv:1610.06726}{http://arxiv.org/abs/1610.06726}.

\bibitem[BHZ16]{BHZ}
\textsc{Y.~{Bruned}}, \textsc{M.~{Hairer}}, and \textsc{L.~{Zambotti}}.
\newblock {Algebraic renormalisation of regularity structures}.
\newblock \emph{ArXiv e-prints} (2016).
\newblock \burlalt{arXiv:1610.08468}{http://arxiv.org/abs/1610.08468}.

\bibitem[CH16]{CH}
\textsc{A.~{Chandra}} and \textsc{M.~{Hairer}}.
\newblock {An analytic BPHZ theorem for regularity structures}.
\newblock \emph{ArXiv e-prints} (2016).
\newblock \burlalt{arXiv:1612.08138}{http://arxiv.org/abs/1612.08138}.

\bibitem[FG16]{FGub}
\textsc{M.~Furlan} and \textsc{M.~Gubinelli}.
\newblock {Paracontrolled quasilinear SPDEs}.
\newblock \emph{ArXiv e-prints} (2016).
\newblock \burlalt{arXiv:1610.07886}{http://arxiv.org/abs/1610.07886}.

\bibitem[GH17]{GH17}
\textsc{M.~{Gerencs\'er}} and \textsc{M.~{Hairer}}.
\newblock {Singular SPDEs in domains with boundaries}.
\newblock \emph{ArXiv e-prints} (2017).
\newblock \burlalt{arXiv:1702.06522}{http://arxiv.org/abs/1702.06522}.

\bibitem[Hai14]{H0}
\textsc{M.~Hairer}.
\newblock A theory of regularity structures.
\newblock \emph{Inventiones mathematicae} \textbf{198}, no.~2, (2014),
  269--504.
\newblock \burlalt{arXiv:1303.5113}{http://arxiv.org/abs/1303.5113}.
\newblock
  \burlalt{doi:10.1007/s00222-014-0505-4}{http://dx.doi.org/10.1007/s00222-014-0505-4}.

\bibitem[{Hai}16]{String}
\textsc{M.~{Hairer}}.
\newblock {The motion of a random string}.
\newblock \emph{ArXiv e-prints} (2016).
\newblock \burlalt{arXiv:1605.02192}{http://arxiv.org/abs/1605.02192}.

\bibitem[HP15]{HP15}
\textsc{M.~Hairer} and \textsc{{\'E}.~Pardoux}.
\newblock A {W}ong-{Z}akai theorem for stochastic {P}{D}{E}s.
\newblock \emph{J. Math. Soc. Japan} \textbf{67}, no.~4, (2015), 1551--1604.
\newblock \burlalt{arXiv:1409.3138}{http://arxiv.org/abs/1409.3138}.
\newblock
  \burlalt{doi:10.2969/jmsj/06741551}{http://dx.doi.org/10.2969/jmsj/06741551}.

\bibitem[OSSW17]{HendrikNew}
\textsc{F.~Otto}, \textsc{J.~Sauer}, \textsc{S.~Smith}, and \textsc{H.~Weber}.
\newblock {Parabolic equations with low-regularity coefficients and rough
  forcing}.
\newblock \emph{In preparation} (2017).

\bibitem[OW16]{OWeb}
\textsc{F.~Otto} and \textsc{H.~Weber}.
\newblock {Quasilinear SPDEs via rough paths}.
\newblock \emph{ArXiv e-prints} (2016).
\newblock \burlalt{arXiv:1605.09744}{http://arxiv.org/abs/1605.09744}.

\end{thebibliography}
\bibliographystyle{Martin} 

\end{document}